\documentclass{birkjour}
 \newtheorem{thm}{Theorem}[section]
 \newtheorem{cor}[thm]{Corollary}
 \newtheorem{lem}[thm]{Lemma}
 \newtheorem{prop}[thm]{Proposition}
 \theoremstyle{definition}
 \newtheorem{defn}[thm]{Definition}
 \theoremstyle{remark}

 \numberwithin{equation}{section}

\begin{document}

%
%
%
%
%
%
%
%
%

\title[Heisenberg's  and Hardy's Uncertainty Principles  ]
 {Heisenberg's  and Hardy's Uncertainty \\Principles   in   Real Clifford Algebras}


\author[R. Jday]{Rim Jday}
\address{ University of Tunis El Manar\br
 Faculty of Sciences Tunis\br 
  Laboratory Special
  Functions\br
  Harmonic Analysis
   and Analogue\br
 2092 Tunis\br
 }
\email{rimjday@live.fr}


\date{Decembre 31, 20016}

\begin{abstract}
Recently, many surveys are devoted to study the Clifford Fourier transform. Dealing with the real Clifford Fourier transform introduced by Hitzer \cite{9}, we establish  analogues of the classical Heisenberg's inequality and Hardy's theorem in the real Clifford algebra $Cl(p,q)$. 
\end{abstract}

\maketitle
\section{Introduction}
According to the relevance of the Clifford Fourier transform in applied mathematics, many works focus  on the study of Fourier transform in the framework of Clifford algebras. In \cite{1}, the Fourier transform is extended to a multivector valued function-distributions with compact support. Then several definitions of Fourier transform  in Clifford algebras are given such as Quaternion Fourier transform (QFT) \cite{8,7} and Clifford Fourier transform (CFT) \cite{2,3,6}. We refer the reader to \cite{key-12} for more details of  Fourier transforms over Clifford algebras.

Regarding the fundamental property of the complex imaginary unit $j\in\mathbb{C}$, i.e. $j^2=-1$, researchers concentrate  on finding  Clifford numbers satisfying this property (square roots minus one in Clifford algebras) \cite{4,5}. Replacing $j$ by a square root of $-1$ allows to define a generalized Fourier transform in Clifford algebras \cite{6,9} .
Our aim in the present paper is to obtain uncertainty principle for the Clifford Fourier transform  stated by E. Hitzer in \cite{9}. 

The paper is structured  as follows. In section \ref{section2}, we recall necessary background knowledge of Clifford algebras and some results that will be usefull in the sequel. In section \ref{section3},  we adopt the real Clifford Fourier transform of \cite{9} and we review its properties. In section \ref{section4},  we provide Heisenberg's inequality and Hardy's theorem  for the   Clifford Fourier transform.  
\section{Preliminaries}\label{section2}

The Clifford  geometric algebra $Cl(p,q)$ is defined as a non-commutative algebra generated by the basis $\{e_1,e_2,..,e_n\}$, with $n=p+q$, satisfying the multiplication rule:
\begin{equation}
e_ke_l+e_le_k=2\epsilon_k\delta_{k,l},\quad k,l=1,..,n, 
\end{equation}
with $\mathcal{\epsilon}_k=+1$ for $k=1,..,p$ and $\mathcal{\epsilon}_k=-1$ for $ k=p+1,..,n$.
Here $\delta_{k,l}$ denotes the Kronecker symbol, i.e., $\delta_{k,l}=1$  for $k=l$ and  $\delta_{k,l}=0$ for $k\neq l$.\\
This algebra can be decomposed as 
 \begin{equation}
 \displaystyle Cl(p,q)=\bigoplus_{k=0}^n Cl^k(p,q),
  \end{equation}
 with $Cl^k(p,q)$ the space of k-vectors given  by
  \begin{equation}
Cl^k(p,q):={\rm span}\{e_{i_1}..e_{i_k},i_1<..<i_k\}.
  \end{equation}
  Hence  $\left\{e_A=e_{i_1}e_{i_2}\cdots e_{i_k},\, A\subseteq \{1,\cdots,n\}, \, 1\leq i_1<\cdots <i_k\leq n \right\} $ with \\ $e_{\emptyset}=1$,  presents a $2^n$-dimensional graded basis  of $Cl(p,q)$.\\
Every  element  $\alpha$ of $ Cl(p,q)$ (Clifford number) can be written as
\begin{equation}
\alpha=\displaystyle\sum_A \alpha_Ae_A=\left\langle \alpha\right\rangle_0+\left\langle \alpha\right\rangle_1+\cdots+\left\langle \alpha\right\rangle_n,
\end{equation}
with $\alpha_A$ real numbers.
Here 
 $\left\langle \cdot\right\rangle_k$ is the grade $k $-part of $\alpha$. As exemple, $<\cdot>_0$ denotes the  scalar part,  $<\cdot>_1$ is the vector part and $<\cdot>_2$ is the bivetor part.   \\
 A vector $x\in\mathbb{R}^{p,q}$ can be identified by
 \begin{equation}
 x=\displaystyle\sum_{l=1}^nx_le^l=\sum_{l=1}^nx^le_l, \quad \text{with}\quad e^l:=\epsilon_le_l.
 \end{equation}
  The scalar product and outer product  for  a $k$-vector $A_k\in Cl(p,q)$ and an $s$-vector $B_s\in Cl(p,q)$ are defined, respectively, by 
\begin{equation}
A_k\ast B_s=\left\langle A_k B_s\right\rangle_0,\qquad A_k\wedge B_s=\left\langle A_k B_s\right\rangle_{k+s}.
\end{equation}
In $Cl(p,q)$, the complex  and the quaternion conjugations are replaced by the principle reverse which is introduced  for all $\alpha \in Cl(p,q)$ by 
\begin{equation}
\widetilde{\alpha}=\displaystyle \sum_{k=0}^n (-1)^{\frac{k(k-1)}{2}}\overline{ \left\langle \alpha\right\rangle_k},
\end{equation}
where $\overline{e_A}=\epsilon_{i_1}e_{i_1}\cdots\epsilon_{i_k}e_{i_k}$, $1\leq i_1<i_2<\cdots<i_k\leq n$. In the case of $\mathbb{C}\otimes Cl(p,q)$, we add the rule $\widetilde{j}=-j$ with $j\in \mathbb{C}$ is the complex imaginary unit.\\
\noindent Note that  
$\widetilde{e_A}\ast e_B=\delta_{A,B},\, \text{for all}\quad 1\leq A,B\leq 2^n$. In particular, we have $e^l\ast e_k=e^l\cdot e_k=\delta_{k,l}$.\\
This leads to
\begin{equation}\label{geometric}
 M\ast \widetilde{N}=\displaystyle\sum_A M_AN_A,\quad \forall \, M,N\in Cl(p,q).
 \end{equation}
The modulus of  a multivector $M\in Cl(p,q)$  is given by  
\begin{equation}
|M|^2=M\ast \widetilde{M}=\displaystyle\sum_A M_A^2,
\end{equation}
from which wa can find Cauchy-Schwartz's inequality for multivectors:
\begin{equation}\label{Cauchy}
|M \ast\widetilde{N}|\leq |M||N|.
\end{equation} 
In the sequel, we will consider function defined on $\mathbb{R}^{p,q}$ and taking values  in $Cl(p,q)$. These functions can be expressed like follows
\begin{equation}
f(x)=\sum_{A}f_A(x)e_A,
\end{equation}
where $f_A$ are real-valued functions.
\begin{defn} Let $f,g: \mathbb{R}^{p,q}\rightarrow Cl(p,q)$. Then inner product of $f$ and $g$ is defined by 
\begin{equation}(f,g)=\displaystyle\int_{\mathbb{R}^{p,q}} f(x)\widetilde{g(x)}d^nx=\displaystyle\sum_{A,B}e_A\widetilde{e_B}\int_{\mathbb{R}^{p,q}}f_A(x)g_B(x)d^nx,
\end{equation}
with symmetric scalar part
\begin{equation}
<f,g>=\displaystyle\int_{\mathbb{R}^{p,q}}f(x)\ast \widetilde{g(x)}d^nx=\displaystyle\sum_A\int_{\mathbb{R}^{p,q}}f_A(x)g_A(x)d^nx,
\end{equation}
which induces the following  $L^2(\mathbb{R}^{p,q},Cl(p,q))$-norm
\begin{equation}\label{normf}
||f||^2:=<(f,f)>=\displaystyle \int_{\mathbb{R}^{p,q}}|f(x)|^2d^nx=\sum_A\int_{\mathbb{R}^{p,q}}f_A^2(x)d^nx.
\end{equation}
Then $L^2(\mathbb{R}^{p,q},Cl(p,q))$ is introduced as
\begin{equation}\displaystyle\qquad L^2(\mathbb{R}^{p,q},Cl(p,q))=\left\{f:\mathbb{R}^{p,q}\rightarrow Cl(p,q)|\quad ||f||<\infty\right\}.\end{equation}
\end{defn}
\noindent The vector derivative $\nabla$ of a function $f$ is defined by 
\begin{equation}
\nabla=\displaystyle\sum_{l=1}^ne^l\partial_l \quad \text{with}\quad \partial_l=\frac{\partial}{\partial_{x_l}}, \quad 1\leq l\leq n. 
\end{equation}
For an arbitrary vector $a\in \mathbb{R}^{p,q}$, the vector differential in the $a$-direction is given by 
\begin{equation}
a\cdot \nabla f(x)=\lim\limits_{\epsilon\rightarrow 0}\frac{f(x+\epsilon a)-f(x)}{\epsilon}.
\end{equation}

\begin{prop}[Integration of parts]\label{integration} \cite[Proposition 3.9]{6}
$$\displaystyle\int_{\mathbb{R}^n}f(x)[a \cdot \nabla g(x)]d^nx=\left[\int_{\mathbb{R}^{n-1}}f(x)g(x)d^{n-1}x\right]_{a\cdot x=-\infty}^{a\cdot x=\infty}-\int_{\mathbb{R}^n}[a\cdot \nabla f(x)] g(x)d^nx.$$
\end{prop}
 Recently, an interest is given to study square roots minus one  in the real Clifford algebra, i.e. Clifford numbers $i \in Cl(p,q)$ such that $i^2=-1$, since it generalize the complex imaginary unit. Note that every multivector with respect to any square roots can  split into commuting  and anticommuting parts.
\begin{lem}{\rm\cite{5}} Let $A\in Cl(p,q)$. Then with respect to a square root $i\in Cl (p,q)$ of $-1$, $A$ has a unique decomposition 
$$A=A_{+i}+A_{-i},$$
with 
$$A_{+i}=\frac{1}{2}\left(A+i^{-1}Ai\right)\quad {\rm and}\qquad A_{-i}=\frac{1}{2}\left(A-i^{-1}Ai\right).$$
Moreover one has
$$ A_{+i}i=iA_{+i} \quad{\rm and}\quad A_{-i}i=-iA_{-i}.\quad$$
\end{lem}
\section{Clifford-Fourier transform}\label{section3}
In this section, we recall the real Clifford Fourier transform and its properties. For more details, we refer the reader to \cite{9}. Furthermore, we add some results related to the Clifford Fourier transform and its kernel.
\begin{defn} Let $f\in L^1(\mathbb{R}^{p,q},Cl(p,q))$ and $i\in Cl(p,q)$ be a square root of $-1$. The Clifford-Fourier transform (CFT), with respect to $i$, is defined by 
\begin{equation}
{\mathcal{F}}^i\{f\}(w)= \displaystyle\int_{\mathbb{R}^{p,q}}f(x)e^{-iu(x,w)} d^nx,
\end{equation}
where $d^nx=dx_1\cdots dx_n$, $x,w\in \mathbb{R}^{p,q}$ and  $u: \mathbb{R}^{p,q}\times \mathbb{R}^{p,q}\rightarrow \mathbb{R}$. 
\end{defn}
\noindent Due to the non-commutativity of $Cl(p,q)$, we have right and left linearity for the Fourier transform defined above:
\begin{equation}
{\mathcal{F}}^i\{ \alpha h_1+\beta h_2\}(w)=\alpha {\mathcal{F}}^i\{h_1\}(w)+\beta {\mathcal{F}}^i\{h_2\}(w),\qquad
\end{equation}
\begin{equation}
{\mathcal{F}}^i\{h_1\alpha + h_2\beta\}(w)={\mathcal{F}}^i\{h_1\}(w)\alpha_{+i}+{\mathcal{F}}^{-i}\{h_1\}(w)\alpha_{-i}
\end{equation}
$$\qquad\qquad\qquad\qquad\quad +{\mathcal{F}}^i\{h_2\}(w)\beta_{+i}+{\mathcal{F}}^{-i}\{h_2\}(w)\beta_{-i},$$
for all $h_1,h_2 \in  L^1(\mathbb{R}^{p,q},Cl(p,q))$ and $\alpha ,\beta\in Cl(p,q)$.\\
In the rest of this work, we  will assume that 
\begin{equation}
u(x,w)=x\ast \widetilde{w}= \displaystyle\sum_{l=1}^n x^lw^l= \sum_{l=1}^n x_lw_l\quad
\end{equation}
and $i=\displaystyle\sum_{A}i_Ae_A$ be a square root $-1$ satisfying
\begin{equation}\label{cond}
\qquad \widetilde{i}=-i.
\end{equation}
The Clifford Fourier transform admits the following properties: 
\begin{itemize}
\item[i)]Inversion formula: for $h,{\mathcal{F}}^i\{h\}\in L^1(\mathbb{R}^{p,q},Cl(p,q))$, we have 
\begin{equation}\label{inversion}
h(x)=\mathcal{F}_{-1}^i\{\mathcal{F}^i\{h\}\}(x)= \frac{1}{(2\pi)^n}\displaystyle \int_{\mathbb{R}^{p,q}}\mathcal{F}^i\{h\}(w)e^{iu(x,w)}d^nw,
\end{equation}
where $d^nw=dw_1\cdots dw_n$ and $x,w\in \mathbb{R}^{p,q}$.
\item[ii)] Plancherel identity: for   $h_1, h_2\in L^2(\mathbb{R}^{p,q},Cl(p,q))$,
\begin{equation}
<h_1,h_2>=\frac{1}{(2\pi)^n}<\mathcal{F}^i\{h_1\},\mathcal{F}^i\{h_2\}>. 
\end{equation}
\item[iii)] Parseval identity:  for  $h\in L^2(\mathbb{R}^{p,q},Cl(p,q))$, we have
\begin{equation}\label{parseval}
||h||=\frac{1}{(2\pi)^n} ||\mathcal{F}^i\{h\}|| . 
\end{equation}
\item[4i)] Scaling  property: for $a\in\mathbb{R}\setminus\{0\}$,
\begin{equation}\label{scaling}
\mathcal{F}^i\{h_d\}(w)=\frac{1}{|a|^n}\mathcal{F}^i\{h\}(\frac{1}{a}w),
\end{equation}
with  $h_d(x):=h(ax),\, x\in \mathbb{R}^{p,q}$.
\end{itemize}
\begin{prop}\label{Gauss}
The Gaussian function $f(x)=e^{-\frac{|x|^2}{2}}$  on $\mathbb{R}^n$ defines an eigenfunction for the Clifford Fourier transform:
\begin{equation}
\mathcal{F}^i\{ f\}(w)= (2\pi)^\frac{n}{2}f(w).
\end{equation}
\end{prop}
\begin{prop}\label{eigenfunction} Let $t>0$. Then 
\begin{equation}
\mathcal{F}^i\{e^{-t|.|^2}\}(w)=(2\pi)^\frac{n}{2}\frac{1}{(2t)^\frac{n}{2}}e^{-\frac{|w|^2}{4t}}.
\end{equation}
\end{prop}
\begin{proof}
An application of \eqref{scaling} and Proposition \ref{Gauss} yields the desired result. 
\end{proof}
\begin{prop}\label{proposition2}
\begin{equation}
\mathcal{F}^i\{a\cdot\nabla f\}(w)=a\cdot w\mathcal{F}^i\{f\}(w)i
\end{equation}
\begin{proof}
Following  \cite[Corrolary 3.13]{6} with substituting $i_n$ by $i$, we  can show  that
\begin{equation} a\cdot\nabla e^{iu(x,w)}=a\cdot wie^{iu(x,w)}.
\end{equation}
Thus we have 
$$\quad a\cdot\nabla f(x)=a\cdot\nabla\frac{1}{(2\pi)^n}\displaystyle \int_{\mathbb{R}^{p,q}}\mathcal{F}^i\{f\}(w)e^{iu(x,w)}d^nw$$
$$\qquad\qquad\qquad=\frac{1}{(2\pi)^n}\displaystyle \int_{\mathbb{R}^{p,q}}\mathcal{F}^i\{f\}(w)\left(a\cdot\nabla e^{iu(x,w)}\right)d^nw$$
$$\qquad\qquad\qquad=\frac{1}{(2\pi)^n}\displaystyle \int_{\mathbb{R}^{p,q}}\mathcal{F}^i\{f\}(w)\left(a\cdot wie^{iu(x,w)}\right)d^nw$$
$$=\mathcal{F}_{-1}^i\{a\cdot w\mathcal{F}^i\{f\}(w)i\}.\qquad\quad$$
Using inversion formula, we get 
$$\mathcal{F}^i\{a\cdot\nabla f\}(w)=a\cdot w\mathcal{F}^i\{f\}(w)i.$$ 
\end{proof}
\end{prop}
\begin{thm}\label{noyau} Let  $z=a+jb\in \mathbb{C}\otimes \mathbb{R}^{p,q}$ with  $a,b\in \mathbb{R}^{p,q}$  and $j\in \mathbb{C}$ is the complex imaginary unit.
Put $u(x,z)=x\ast \widetilde{z}$,  $x\in \mathbb{R}^{p,q}$. Then we have

\begin{equation}|e^{-iu(x,z)}|\leq (1+|i|^2)^{\frac{1}{2}} e^{|x||b|}.\end{equation}
\end{thm}
\begin{proof}
Observe  that $$u(x,z)=x\ast \widetilde{z}=x\ast\widetilde{a}+\widetilde{j}(x\ast\widetilde{b})$$$$=x\ast\widetilde{a}-j(x\ast\widetilde{b})$$$$\quad=u(x,a)-ju(x,b).$$
So we compute 
\begin{equation}\label{hardy}
ju(x,z)=u(x,b)+ju(x,a).
\end{equation} 
We should notice that 
$$e^{-iu(x,z)}=\cos(u(x,z))-i \sin(u(x,z)),$$
where 
$$\cos(u(x,z))=\frac{e^{ju(x,z)}+e^{-ju(x,z)}}{2}$$
and 
$$\sin(u(x,z))=\frac{e^{ju(x,z)}-e^{-ju(x,z)}}{2j}.$$
By \eqref{hardy} it follows that 
$$|\cos(u(x,z))|\leq \frac{|e^{ju(x,z)}|+|e^{-ju(x,z)}|}{2}\qquad\quad\qquad$$
$$\qquad\qquad\qquad\qquad\leq  \frac{|e^{u(x,b)}||e^{ju(x,a)}|+|e^{-u(x,b)}||e^{-ju(x,a)}|}{2}$$
$$\quad\leq  \frac{|e^{u(x,b)}|+|e^{-u(x,b)}|}{2}\,.$$
Similarly we obtain
$$|\sin(u(x,z))|\leq  \frac{|e^{u(x,b)}|+|e^{-u(x,b)}|}{2}.$$
See that for all $\gamma\in\mathbb{R}$, 
$$\displaystyle\frac{e^{-\gamma}+e^{\gamma}}{2}\leq e^{|\gamma|}.$$
So we deduce that 
$$|\cos(u(x,z))|\leq e^{|u(x,b)|} $$
and 
$$|\sin(u(x,z))|\leq e^{|u(x,b)|} .$$
Note that 
$$|e^{-iu(x,z)}|^2=|\cos(u(x,z))|^2+\displaystyle\sum_{A}|i_A|^2| \sin(u(x,z))|^2.$$
 Thus
 $$|e^{-iu(x,z)}|\leq \left(1+|i|^2\right)^{\frac{1}{2}}e^{|u(x,b)|}.$$
 Using \eqref{Cauchy}, we get
 $$|u(x,b)|=|x\ast b|\leq |x||b|.$$
 This leads to 
 $$|e^{-iu(x,z)}|\leq \left(1+|i|^2\right)^{\frac{1}{2}}e^{|x||b|}.$$ 
\end{proof}
\section{ Uncertainty Principles}\label{section4}
 In this section, we study uncertainty principles for the  Clifford Fourier transform stated in section 3. Actually, we establish Heisenberg's inequality and Hardy's theorem  in the setting of the real Clifford algebras $Cl(p,q)$. 
\subsection{Heisenberg's Inequality} 
We remind that Heisenberg's inequality was given for $n=3$ in \cite{11} and for $n=2,3( mod(4))$ in \cite{6}. We will prove Heisenberg's inequality  for  the real Clifford Fourier transform in $Cl(p,q)$ with a similar way. 
\begin{thm}[Directional Uncertainty Principle]\label{principle} Let $f\in L^2(\mathbb{R}^{p,q},Cl(p,q))$. Assume $F=\int_{\mathbb{R}^{p,q}}|f(x)|^2d^nx $. Then 
\begin{equation}
\displaystyle\int_{\mathbb{R}^{n}}(a\cdot x)^2|f(x)|^2d^nx\frac{1}{(2\pi)^n}\int_{\mathbb{R}^{n}}(b\cdot w)^2|\mathcal{F}^i\{f\}(w)|^2d^nw\geq (a\cdot\widetilde{b})^2\frac{ 1}{4}F^2,
\end{equation}
with $a,b$ arbitrary constant vectors.
\end{thm}
\begin{proof} Using \eqref{normf}, \eqref{cond}  and  Proposition \ref{proposition2}, we obtain
$$\int_{\mathbb{R}^{n}}(b\cdot w)^2|\mathcal{F}^i\{f\}(w)|^2d^nw=<(b\cdot w\mathcal{F}^i\{f\}(w)i,b\cdot w\mathcal{F}^i\{f\}(w)i)>\quad\quad$$ 
$$\qquad\qquad\qquad\qquad\quad= <(\mathcal{F}^i\{b\cdot\nabla f\}(w),\mathcal{F}^i\{b\cdot\nabla f\}(w))>$$
$$\qquad\qquad\quad=\int_{\mathbb{R}^{n}}|\mathcal{F}^i\{b\cdot\nabla f\}(w)|^2d^nw.$$
An application of Parseval identity (see \eqref{parseval}) leads to 
\begin{equation}
\label{mir}\int_{\mathbb{R}^{n}}(b\cdot w)^2|\mathcal{F}^i\{f\}(w)|^2d^nw=
(2\pi)^n\int_{\mathbb{R}^{n}}|b\cdot\nabla f(w)|^2d^nw.
\end{equation}
See that for $\phi$, $\psi:\mathbb{R}^n\rightarrow \mathbb{C}$,
$$\int_{\mathbb{R}^{n}}|\phi(x)|^2d^nx\int_{\mathbb{R}^{n}}|\psi(x)|^2d^nx\geq \left(\int_{\mathbb{R}^{n}}\phi(x)\overline{\psi(x)}d^nx\right)^2. $$
Thus it follows that 
$$\displaystyle I:=\int_{\mathbb{R}^{n}}(a\cdot x)^2|f(x)|^2d^nx\frac{1}{(2\pi)^n}\int_{\mathbb{R}^{n}}(b\cdot w)^2|\mathcal{F}^i\{f\}(w)|^2 d^nw\qquad\qquad\qquad$$$$\qquad\qquad\qquad=\int_{\mathbb{R}^{n}}(a\cdot x)^2|f(x)|^2d^nx\int_{\mathbb{R}^{n}}|b\cdot\nabla f(w)|^2d^nw$$
$$\qquad\quad\geq\left(\int_{\mathbb{R}^{n}}(a\cdot x)|f(x)||b\cdot\nabla f(x)|d^nx\right)^2.$$
By \eqref{Cauchy} we get
$$I\geq\left(\int_{\mathbb{R}^{n}}(a\cdot x)|\widetilde{f(x)}\ast b\cdot\nabla f(x)|d^nx\right)^2$$
$$\qquad\geq \left(\int_{\mathbb{R}^{n}}(a\cdot x)\left(\widetilde{f(x)}\ast b\cdot\nabla f(x)\right)d^nx\right)^2.$$ 
Since $2\widetilde{f(x)}\ast b\cdot\nabla f(x)=b\cdot\nabla|f(x)|^2$, then
$$I\geq \frac{1}{4} \left(\int_{\mathbb{R}^{n}}(a\cdot x)\left(b\cdot\nabla|f(x)|^2\right)d^nx\right)^2.$$ 
Proposition \ref{integration} yields
$$\qquad I\geq\frac{1}{4}\left(\left[\int_{\mathbb{R}^{n-1}}(a\cdot x)|f(x)|^2d^{n-1}x\right]_{b\cdot x=-\infty}^{b\cdot x=\infty}-\int_{\mathbb{R}^n}[b\cdot\nabla(a\cdot x)] |f(x)|^2d^nx\right)^2$$
$$\geq\frac{1}{4}\left(0-\int_{\mathbb{R}^n} a\cdot\widetilde{b} |f(x)|^2d^nx\right)^2=\frac{1}{4} (a\cdot\widetilde{b})^2 F^2.\qquad\qquad\qquad\qquad$$
\end{proof}
\begin{cor}[Uncertainty Principle] We have 
\begin{equation}
\displaystyle\int_{\mathbb{R}^{n}}(a\cdot x)^2|f(x)|^2d^nx\frac{1}{(2\pi)^n}\int_{\mathbb{R}^{n}}(a\cdot w)^2|\mathcal{F}^i\{f\}(w)|^2d^nw\geq \frac{ 1}{4}F^2,\end{equation}
with equality when $f(x)=C_0e^{-k|x|^2}$, $C_0\in Cl(p,q)$ is an arbitrary constant multivector and $0<k\in\mathbb{R}$.
\end{cor}
\begin{proof}
Put  $b=\pm a$ and $|a|^2=1$. The desired inequality is a straightforward consequence  of Theorem \ref{principle}. Now we move to prove the equality.\\
 For $f=C_0e^{-k|x|^2}$ with $C_0\in Cl(p,q)$, observe that
 $$a\cdot\nabla f=-2ka\cdot x f.$$
Combining this result with \eqref{mir}, we obtain 
$$\displaystyle J:=\int_{\mathbb{R}^{n}}(a\cdot x)^2|f(x)|^2d^nx\frac{1}{(2\pi)^n}\int_{\mathbb{R}^{n}}(a\cdot w)^2|\mathcal{F}^i\{f\}(w)|^2d^nw\qquad\qquad\qquad\qquad\qquad$$$$\qquad\qquad\qquad\qquad=\int_{\mathbb{R}^{n}}(a\cdot x)^2|f(x)|^2d^nx\int_{\mathbb{R}^{n}}|a\cdot\nabla f(w)|^2d^nw$$
$$\qquad\quad=4k^2\left(\int_{\mathbb{R}^{n}}(a\cdot x)^2|f(x)|^2d^nx\right)^2$$
$$\qquad\qquad\quad=4k^2\left(\int_{\mathbb{R}^{n}}a\cdot x a\cdot xf(x)\ast\widetilde{f(x)}d^nx\right)^2$$
$$\qquad\qquad\quad=\left(\int_{\mathbb{R}^{n}}a\cdot x( a\cdot\nabla f(x)\ast\widetilde{f(x)})d^nx\right)^2.$$
By the fact that  $2\widetilde{f(x)}\ast a\cdot\nabla f(x)=a\cdot\nabla|f(x)|^2$ and  Proposition \ref{integration}, we get
$$J=\frac{1}{4}\left(\int_{\mathbb{R}^{n}}a\cdot x a\cdot\nabla|f(x)|^2d^nx\right)^2$$
$$\quad=\frac{1}{4}\left(\int_{\mathbb{R}^{n}}a\cdot\nabla a\cdot x |f(x)|^2d^nx\right)^2$$
$$=\frac{1}{4}|a|^2F^2.\qquad\qquad\qquad\quad$$
Since $|a|^2=1$, we conclude the proof.
\end{proof}
\begin{thm}\label{theoremhei} For $a\cdot\widetilde{b}=0$, we find
\begin{equation}
\displaystyle\int_{\mathbb{R}^{n}}(a\cdot x)^2|f(x)|^2d^nx\frac{1}{(2\pi)^n}\int_{\mathbb{R}^{n}}(b\cdot w)^2|\mathcal{F}^i\{f\}(w)|^2d^nw\geq 0.\end{equation}
\end{thm}
\begin{proof} See that for $a\cdot\widetilde{b}=0$, the right side of Theorem \ref{principle}'s inequality is $0$.
\end{proof}
\begin{thm} One has
\begin{equation}\displaystyle\int_{\mathbb{R}^{n}}|x|^2|f(x)|^2d^nx\frac{1}{(2\pi)^n}\int_{\mathbb{R}^{n}}|w|^2|\mathcal{F}^i\{f\}(w)|^2d^nw\geq n\frac{ 1}{4}F^2.
\end{equation}
\end{thm}
\begin{proof}
Note that $|x|^2=\displaystyle\sum_{k=1}^n(e_k\cdot x)^2$ and $|w|^2=\displaystyle\sum_{k=1}^n(e_k\cdot w)^2$.\\
We compute 
$$K:=\displaystyle\int_{\mathbb{R}^{n}}|x|^2|f(x)|^2d^nx\frac{1}{(2\pi)^n}\int_{\mathbb{R}^{n}}|w|^2|\mathcal{F}^i\{f\}(w)|^2d^nw\qquad\qquad\qquad$$
$$=\displaystyle\sum_{k,l=1}^n (e_k\cdot x)^2|f(x)|^2d^nx\frac{1}{(2\pi)^n}\int_{\mathbb{R}^{n}}(e_l\cdot w)^2|\mathcal{F}^i\{f\}(w)|^2d^nw\qquad$$
$$=\displaystyle\sum_{k=1}^n (e_k\cdot x)^2|f(x)|^2d^nx\frac{1}{(2\pi)^n}\int_{\mathbb{R}^{n}}(e_k\cdot w)^2|\mathcal{F}^i\{f\}(w)|^2d^nw\qquad$$
$$\quad+\sum_{k\neq l}^n (e_k\cdot x)^2|f(x)|^2d^nx\frac{1}{(2\pi)^n}\int_{\mathbb{R}^{n}}(e_l\cdot w)^2|\mathcal{F}^i\{f\}(w)|^2d^nw.\quad$$
From Theorem \ref{principle} and Theorem \ref{theoremhei}, we obtain 
$$\qquad\qquad K\geq \displaystyle\sum_{k=1}^n (e_k\cdot x)^2|f(x)|^2d^nx\frac{1}{(2\pi)^n}\int_{\mathbb{R}^{n}}(e_k\cdot w)^2|\mathcal{F}^i\{f\}(w)|^2d^nw$$
$$\geq \displaystyle\sum_{k=1}^n(e_k\cdot\widetilde{e_k})^2\frac{1}{4}F^2,\qquad\qquad\qquad\qquad\qquad\quad $$
which completes the proof.
\end{proof}
\subsection{Hardy's Theorem}
\begin{thm}
Let $ p,q$ be positive constants. Suppose $f$ is a measurable function on $\mathbb{R}^{p,q}$ satisfying the following estimates:
\begin{equation}\label{condition1}
|f(x)|\leq C e^{-p|x|^2},\quad x\in\mathbb{R}^{p,q}
\end{equation}
\begin{equation}\label{condition2}
|\mathcal{F}^i\{f\}(y)|\leq C e^{-q|y|^2},\quad y\in\mathbb{R}^{p,q},
\end{equation}
where $C$ is a positive constant. Then:
\begin{enumerate}
\item If $pq>\frac{1}{4}$, then $f=0$.
\item If  $pq=\frac{1}{4}$, then $f(x)= Ae^{-p|x|^2}$, with $A$ a Clifford constant.
\item If  $pq<\frac{1}{4}$, then there exist  many such functions.
\end{enumerate}
\end{thm}
\begin{proof} According to \eqref{scaling}, we can assume that $p=q$.
For $pq<\frac{1}{4}$, by Proposition \ref{eigenfunction},  the functions $f(x)= Ce^{-t|x|^2}$ for $t\in]p,\frac{1}{4q}[$ and for some Clifford constant $C$ satisfies \eqref{condition1} and \eqref{condition2}.
Since (1) is deduced from (2), it is sufficient to prove (2). Hence we will show that if $p=q=\frac{1}{2}$, then $f(x)= A e^{-\frac{|x|^2}{2}} $ with $A$ a constant.\\
Using Theorem \ref{noyau} and  \eqref{condition1}, we obtain for all $\lambda \in \mathbb{C}\otimes \mathbb{R}^{p,q}$
$$|\mathcal{F}^i\{f\}(\lambda)|\leq \int_{\mathbb{R}^{p,q}}|f(x)e^{-iu(x,\lambda)}|d^nx\qquad$$
$$\qquad\qquad\qquad\quad\quad\leq \left(1+|i|^2\right)^{\frac{1}{2}} \int_{\mathbb{R}^{p,q}}|f(x)|e^{|x||Im(\lambda)|}d^nx$$
$$\qquad\qquad\qquad\quad\quad\quad\leq C\left(1+|i|^2\right)^{\frac{1}{2}} \int_{\mathbb{R}^{p,q}}e^{-\frac{|x|^2}{2}}e^{|x||Im(\lambda)|}d^nx$$
$$\leq C'e^{\frac{|Im(\lambda)|^2}{2}} ,\qquad\quad$$
\begin{equation}\label{hardy2}
\leq C'e^{\frac{|\lambda|^2}{2}},\qquad\quad\quad
\end{equation}
with $C'$ is a positive constant.\\
Furthermore we get $\mathcal{F}^i\{f\}$ is an entire function satisfying \eqref{hardy2} and \eqref{condition2}. Thus  \cite[Lemma.2.1]{10} implies that 
$$\mathcal{F}^i\{f\}(x)=Ae^{-\frac{|x|^2}{2}},\quad \text{for some constant} \quad A.$$
By \eqref{inversion} and Proposition \ref{Gauss}, it follows that $f(x)=A e^{-\frac{|x|^2}{2}}$.

\end{proof}

\end{document}